\documentclass{amsart}
\usepackage{amscd,amssymb,amsopn,amsmath,amsthm,graphics,amsfonts,accents,enumerate,verbatim,calc}
\usepackage[dvips]{graphicx}
\usepackage[colorlinks=true,linkcolor=red,citecolor=blue]{hyperref}
\usepackage[all]{xy}
\usepackage{cite}
\usepackage{tikz}
\usepackage{tikz-cd}
\usepackage{mathrsfs}

\calclayout

\newcommand{\rt}{\rightarrow}
\newcommand{\lrt}{\longrightarrow}

\newcommand{\va}{\varphi}
\newcommand{\st}{\stackrel}

\newcommand{\al}{\alpha}

\newcommand{\La}{\Lambda}

\newcommand{\Om}{\Omega}

\newcommand{\Z}{\mathbb{Z}}

\newcommand{\SA}{\mathscr{A}}

\newcommand{\SD}{\mathscr{D}}

\newcommand{\SG}{\mathscr{G}}

\newcommand{\SM}{\mathscr{M}}

\newcommand{\SX}{\mathscr{X}}
\newcommand{\SY}{\mathscr{Y}}
\newcommand{\SW}{\mathscr{W}}

\newcommand{\CA}{\mathcal{A} }

\newcommand{\CX}{\mathcal{X} }

\newcommand{\Mod}{{\rm{Mod\mbox{-}}}}

\newcommand{\mmod}{{\rm{{mod\mbox{-}}}}}

\newcommand{\prj}{{\rm{prj}\mbox{-}}}

\newcommand{\Ker}{{\rm{Ker}}}

\newcommand{\rad}{{\rm{rad}}}

\newcommand{\Ext}{{\rm{Ext}}}

\theoremstyle{plain}
\newtheorem{theorem}{Theorem}[section]
\newtheorem{corollary}[theorem]{Corollary}

\newtheorem{proposition}[theorem]{Proposition}

\theoremstyle{definition}
\newtheorem{definition}[theorem]{Definition}

\newtheorem{remark}[theorem]{Remark}

\newtheorem{s}[theorem]{}

\theoremstyle{plain}

\theoremstyle{definition}

\numberwithin{equation}{section}

\begin{document}

\title[Higher Cotorsion Classes]{Cotorsion Classes in Higher Homological Algebra}

\author[J. Asadollahi, A. Mehregan and S. Sadeghi]{Javad Asadollahi, Azadeh Mehregan and Somayeh Sadeghi}

\address{Department of Pure Mathematics, Faculty of Mathematics and Statistics, University of Isfahan, P.O.Box: 81746-73441, Isfahan, Iran}
\email{asadollahi@sci.ui.ac.ir, asadollahi@ipm.ir }
\email{so.sadeghi@sci.ui.ac.ir }
\email{azadeh.mehregan@sci.ui.ac.ir }



\begin{abstract}
In this note, the notion of cotorsion classes is introduced into the higher homological algebra. Our results motivate the definition, showing that this notion of $n$-cotorsion classes satisfies usual properties one could expect. In particular, a higher version of Wakamatsu's Lemma is proved. Connections with wide subcategories are also studied.
\end{abstract}

\maketitle

%\tableofcontents

\section{Introduction}
Let $n \geq 1$ be a fixed integer. Higher homological algebra, known also as $n$-homological algebra, is a recently discovered generalisation of homological algebra, having sequences of length $n+2$ playing the role of the short exact sequences in ($1$-)abelian categories. It appeared in a series of papers by Iyama \cite{I1, I2, I3} and then axiomatized and studied extensively by Jasso \cite{Ja}.

Investigating possible definitions of various notions of classical homological algebra in the context of higher homological algebra is an active research area, see for instance, \cite{Jor}, \cite{HJV}, \cite{JJ} and \cite{EN}, where notions such as torsion classes, wide subcategories, abelian quotients of triangulated categories and Auslander's Formula are studied in the higher context, respectively.

A notion of interest in classical homological algebra is the concept of a cotorsion theory, introduced in \cite{S}. Let $\SA$ be an abelian category with enough projective and enough injective objects. A cotorsion pair is a pair $(\SX, \SY)$ of full subcategories of $\CA$  such that
\[\SX^{\perp_1}=\SY  \ \ \ {\rm and} \ \ \ \SX={}^{\perp_1}\SY,\]
where orthogonals are taken with respect to the functor $\Ext^1$. $\SX$ is then called a cotorsion class. According to their applications, cotorsion pairs have been studied extensively in the literature, see e.g. \cite{X}, \cite{BR}, \cite{EJ} and \cite{GT}.

Our aim in this paper is to investigate a higher version of the notion of cotorsion classes.  This investigation is interesting in its own right and also may shed some light in some problems in higher homological algebra, in particular, in higher tilting theory.

We introduce the notion of $n$-cotorsion class $\SX$ of an $n$-cluster titling subcategory $\SM$ of an abelian category $\SA$. We are mainly interested in the case where $\SM$ is an $n$-cluster tilting subcategory of $\mmod\La$, where $\La$ is an artin algebra, that is, in the case where $(\La, \SM)$ is an $n$-homological pair \cite[Definition 2.5]{HJV}.

It is immediate that every cotorsion class $\SX$ is closed under finite sums, summands and contains projective objects. Moreover, cotorsion classes have tie connections with the notion of special precovering classes. We show that $n$-cotorsion classes also enjoy similar properties, i.e. are closed under sums, summands and contain projectives and also are related with the $n$-special precovering classes, see Theorem \ref{NspecialNcotor}.

A full subcategory $\SX$ of an $n$-cluster tilting subcategory $\SM$ of an abelian category $\SA$ is called an $n$-special precovering class if for every $m\in\SM$, there exists an $n$-exact sequence \[0 \lrt r_n\lrt \cdots\lrt r_1\lrt x\lrt m\lrt 0\] in $\SM$ such that $x\in\SX$ and for every $x' \in \SX$ the induced sequence
\[ 0 \lrt \Ext^n_{\SA}(x', r_n) \lrt \Ext^n_{\SA}(x', r_{n-1}) \lrt \cdots \lrt \Ext^n_\SA(x', r_1) \lrt 0\]
of abelian groups is exact, see Definition \ref{n-special precover}.

Wakamatsu's Lemma \cite[Lemma 2.1.13]{GT}, which is one of the pillars in the classical homological algebra, states that if $\varphi: x \lrt a$ in an abelian category $\SA$, is an $\SX$-cover, then $\Ker\varphi \in \SX^{{\perp}_1}$, provided $\SX$ is closed under extensions, see next section for the definition of an $\SX$-cover. Here using the notion of $n$-cotorsion classes we provide a higher version of Wakamatsu's Lemma in higher homological algebra, see Theorem \ref{Wakamatsu}.

In the last section of the paper, using the concept of wide subcategories \cite{HJV}, we study $n$-cotorsion classes. It will be shown that if $\SW$ is a wide subcategory of an $n\Z$-cluster tilting subcategory $\SM$ containing projective modules and enjoying the property that the inclusion functor $\SW\st{i} \lrt \SM$ admits a right adjoint, then $\SW$ is an $n$-cotorsion class of $\SM$, see Theorem \ref{Wide}.

\section{Preliminaries}
Throughout the paper, $\La$ is an Artin algebra and $\mmod\La$ is the category of finitely generated right $\La$-modules.

Let $\SA$ be an additive category and $\SX$ be a subclass of $\SA$. Let $a$ be an object of $\SA$. An $\SX$-precover of $a$ is a morphism $x\st{\varphi}\lrt a$ in $\SA$ with $x\in\SX$ such that any other morphism $x'\lrt a$ with $x'\in\SX$ factors through $\varphi$. If every object in $\SA$ admits an $\SX$-precover, then $\SX$ is called a precovering class of $\SA$. An $\SX$-precover $x\st{\varphi}\lrt a$ is called an $\SX$-cover if every endomorphism $\psi: x\lrt x$  with $\varphi\psi=\varphi$ is an automorphism. $\SX$ is called covering if every $a\in\SA$ admits an $\SX$-cover. $\SX$-preenvelopes, preenveloping and enveloping classes are defined dually.

Note that there are other terminologies for the above concepts, used mostly in representation theory of algebras: for an object $a$ of $\SA$, $\SX$-precover of $a$, resp. $\SX$-preenvelope of $a$, is called right $\SX$-approximation, resp. left $\SX$-approximation, of $a$. Moreover, precovering classes, resp. preenveloping classes, are called contravariantly finite, resp. covariantly finite, subcategories of $\SA$. In this context, a subcategory which is both contravariantly finite and covariantly finite (or equivalently, precovering and preenveloping) is called functorially finite.

A subcategory $\SG$ of $\SA$ is called a generating subcategory if for every object $a \in \SA$, there exists an epimorphism $m \lrt a$ with $m \in \SG$.
Cogenerating subcategories are defined dually. $\SG$ is called a generating-cogenerating subcategory of $\SA$ if it is both a generating and a cogenerating subcategory.

\s{\sc $n$-Cluster tilting subcategories.} Let $\SA$ be an abelian category and $n\geq 1$ be a fixed integer. An additive subcategory $\SM$ of $\SA$ is called an $n$-cluster tilting subcategory if it is a functorially finite and generating-cogenerating subcategory of $\SA$ satisfying $\SM^{\perp_n}=\SM= {}^{\perp_n}\SM$, where
\[\SM^{\perp_n}:= \{a \in \SA \mid \Ext^i_{\SA}(\SM, a)=0, \ \text{ for all} \ 0 < i<  n \}, \]
\[{}^{\perp_n}\SM:= \{a \in \SA \mid \Ext^i_{\SA}(a, \SM)=0, \ \text{ for all} \ 0 < i <  n \}.\]

When $\SA=\mmod\La$, and $\SM$ is an $n$-cluster tilting subcategory of $\SA$, the pair $(\La, \SM)$ is called an $n$-homological pair \cite[Definition 2.5]{HJV}.

Cluster tilting subcategories are defined by Iyama in \cite[Definition 2.2]{I1} and studied further in \cite{I3}; see also \cite[Definition 3.14]{Ja}. It is known that $\SM$ has a structure of an $n$-abelian category  \cite[Theorem 3.16]{Ja}.

{\s{\sc $n$-Abelian categories}} {(see \cite[\S 3]{Ja})}.
Let $\SM$ be an additive category. Let $u_{n+1} : m_{n+1} \lrt m_{n}$  be a morphism in $\SM$. An $n$-cokernel of $u_{n+1}$  is a sequence
\[m_n \st{ u_n }{\lrt } m_{n-1} \lrt \cdots \lrt m_{1}\st{ u_{1} }{ \lrt } m_{0}\]
of morphisms in $\SM$ such that for every $m \in \SM$, the induced sequence
\[0 \lrt \SM(m_0, m) {\lrt} \cdots {\lrt} \SM(m_n, m) {\lrt} \SM(m_{n+1}, m)\]
of abelian groups is exact. $n$-cokernel of $u_{n+1}$ denotes by $(u_n, u_{n-1}, \cdots, u_{1})$. The notion of $n$-kernel of a morphism $u_1:m_1 \lrt m_{0}$ is defined similarly, or rather dually.

A sequence $m_{n+1} \st{u_{n+1} }{\lrt } m_n \lrt  \cdots \lrt m_1 \st{u_1}{\lrt} m_{0}$ of objects and morphisms in $\SM$ is called $n$-exact \cite[Definitions 2.2, 2.4]{Ja} if $(u_{n+1}, u_{n}, \cdots, u_{2})$ is an $n$-kernel of $u_1$ and $(u_n, u_{n-1}, \cdots, u_{1})$ is an $n$-cokernel of $u_{n+1}$. An $n$-exact sequence like the above one, will be denoted by
\[0 \lrt m_{n+1} \st{ u_{n+1} }{ \lrt }m_n \lrt  \cdots \lrt m_1 \st{ u_1 }{\lrt } m_{0} \lrt 0.\]

The additive category $\SM$ is called $n$-abelian \cite[Definition 3.1]{Ja} if it is idempotent complete, each morphism in $\SM$ admits an $n$-cokernel and an $n$-kernel and every monomorphism $u_{n+1}: m_{n+1} {\lrt } m_n$, respectively every epimorphism $u_1: m_1 { \lrt } m_{0}$, can be completed to an $n$-exact sequence
\[0 \lrt m_{n+1} \st{ u_{n+1} }{ \lrt }m_n \lrt  \cdots \lrt m_1 \st{ u_1 }{\lrt } m_{0} \lrt 0.\]

\s Let $\SM$ be an $n$-abelian category and
\[\eta: 0 \lrt m_{n+1} \st{f_{n+1}}\lrt m_n \lrt \cdots \lrt m_2 \st{f_2}\lrt m_1 \st{f_1} \lrt m_0 \lrt 0\]
be an $n$-exact sequence in $\SM$. By \cite[Proposition 2.6]{Ja}, $f_{n+1}$ is a split monomorphism if and only if $f_1$ is a split epimorphism and these are equivalent to the fact that the identity morphism on $\eta$ is null-homotopic. In this case, we say that $\eta$ is a contractible (or split) $n$-exact sequence.

\s Projective and injective objects in an $n$-abelian category $\SM$ are defined in the usual sense. For instance,  $p\in\SM$ is called projective \cite[Definition 3.11]{Ja} if for every epimorphism $f: a\rt b$, the sequence $\SM(p, a)\lrt\SM(p, b)\lrt 0$ is exact. This type of projective objects are called `categorically projective' in
\cite[Definition 4.6]{HJV}. The notion of an injective object is defined dually.

\s By \cite[Definition 2.22]{IJ}, we say that an $n$-abelian category $\SM$ has $n$-syzygies if for every $m \in \SM$ there exists an $n$-exact sequence
\[0\rt  k \lrt   p_n \lrt \cdots \lrt  p_1 \lrt m \lrt 0\]
in $\SM$, such that $p_i$, for $i \in \{1,  \cdots, n\}$, is a projective object. $k$ is then called an $n$-syzygy of $m$ and, by abuse of notation, is denoted by $\Om_nm$. The notion of $n$-cosyzygies and $\SM$ having $n$-cosyzygies are defined dually. The $n$-cosyzygy of $m$ is denoted by $\Om_{-n}m$.

\s {\sc $n$-Pushout diagrams.} Let $\SM$ be an $n$-abelian category. Let 
\[\eta: m_{n+1} \lrt m_n \lrt \cdots \lrt m_1\]
be a complex in $\SM$. An $n$-pushout of $\eta$ along a morphism $f: m_{n+1} \lrt m'_{n+1}$ in $\SM$, is a diagram
\begin{equation*}
\begin{tikzcd}
\eta: & m_{n+1} \rar \dar{f} & m_n \rar \dar & \cdots \rar & m_2 \rar \dar & m_1 \dar\\
& m'_{n+1}  \rar & m'_n \rar & \cdots \rar & m'_2 \rar & m'_1
\end{tikzcd}
\end{equation*}
such that $m'_i \in \SM$, for all $i \in \{1, 2, \cdots, n\}$, and in the mapping cone
\[ m_{n+1} \st{\al}\lrt m_n\oplus m'_{n+1} \st{\al_n}\lrt m_{n-1}\oplus m'_n \st{\al_{n-1}}\lrt \cdots \lrt m_1\oplus m'_2 \st{\al_1}\lrt m'_1   \] 
the sequence $(\al_n, \al_{n-1}, \cdots, \al_1)$ is an $n$-cokernel of $\al$ \cite[Definition 2.11]{Ja}. The notion of $n$-pullback is defined dually. $n$-pushouts and $n$-pullbacks in $\SM$ always exist. Moreover, if $\eta$ completes to an $n$-exact sequence such as
$0 \lrt m_{n+1} \lrt m_n \lrt \cdots \lrt m_1 \lrt m_0 \lrt 0$, then the $n$-pushout diagram also completes to the following diagram
\begin{equation*}
\begin{tikzcd}
0 \rar  & m_{n+1} \rar \dar{f} & m_n \rar \dar & \cdots \rar & m_2 \rar \dar & m_1 \rar \dar & m_0 \rar \dar[equals] & 0 \\
0 \rar & m'_{n+1}  \rar & m'_n \rar & \cdots \rar & m'_2 \rar & m'_1 \rar & m_0 \rar & 0
\end{tikzcd}
\end{equation*}
in which the lower row is also an $n$-exact sequence in $\SM$, see \cite[Theorem 3.8]{Ja}. 

\s {\sc Morphisms of $\Ext^n$ groups in $n$-cluster tilting subcategories.}
Let $(\La, \SM)$ be an $n$-homological pair. Let $m$ and $m'$ be objects of $\SM$. By \cite[\S IV.9]{HS}, every element $\eta$ of $\Ext^n_{\La}(m, m')$ is represented by an exact sequence
\[\eta: 0 \lrt m' \lrt m_n \lrt m_{n-1} \lrt \cdots \lrt m_1 \lrt m \lrt 0\] 
in $\mmod\La$. Since $\SM$ is an $n$-cluster tilting subcategory of $\mmod\La$, by \cite[Appendix A]{I1}, one can assume that all the middle terms $m_i$ also are in $\SM$. Now let $f: m' \lrt m''$ be a morphism in $\SM$. For every element $m \in \SM$, the morphism \[\hat{f}=\Ext^n(m, f): \Ext^n_{\La}(m, m') \lrt \Ext^n_{\La}(m, m'')\] can be interpreted in terms of $n$-pushout of $n$-exact sequences in $\SM$. That is, if \[\eta: \ 0 \lrt m' \lrt m_n \lrt m_{n-1} \lrt \cdots \lrt m_1 \lrt m \lrt 0\] is an element of $\Ext^n_{\La}(m, m')$, then $\hat{f}(\eta)=\eta' \in \Ext^n_{\La}(m, m'')$ is obtained by extending $n$-pushout of $\eta$ along $f$, as depicted in the following diagram
\begin{equation*}
\begin{tikzcd}
\eta: \ \ 0 \rar  & m' \rar \dar{f} & m_n \rar \dar & \cdots \rar & m_2 \rar \dar & m_1 \rar \dar & m \rar \dar[equals] & 0 \\
\eta': \ \ 0 \rar & m''  \rar & m'_n \rar & \cdots \rar & m'_2 \rar & m'_1 \rar & m \rar & 0
\end{tikzcd}
\end{equation*}
To see this one should follow similar argument as in the \cite[pp. 148-150]{HS} and then apply Theorem IV.9.1 of \cite{HS}. See also \cite[Remark 3.8(b) and Lemma 3.13]{Fe}.

\s{\sc $n\Z$-Cluster tilting subcategories.}
Let $(\La, \SM)$ be an $n$-homological pair. By \cite[Lemma 3.5]{I3}, for each $m \in \SM$ and for each exact sequence
\[ 0\lrt m_{n+1} \lrt  m_n \lrt m_{n-1} \lrt \cdots \lrt  m_1 \lrt m_0 \lrt 0 \]
in $\mmod\La$ with all terms in $\SM$, i.e. each $n$-exact sequence in $\SM$, there exists an exact sequence
\[0 \rt \SM(m, m_{n+1}) \rt \SM(m, m_n) \rt \cdots \rt \SM(m, m_0) \rt \Ext^n_\La(m, m_{n+1}) \rt \Ext^n_\La(m, m_n)\]
of abelian groups.

$\SM$ is called $n\Z$-cluster tilting subcategory of $\mmod\La$ if for each $m \in \SM$ and for each $n$-exact sequence
\[ 0\lrt m_{n+1} \lrt  m_n \lrt m_{n-1} \lrt \cdots \lrt  m_1 \lrt m_0 \lrt 0 \]
in $\SM$, there exists an exact sequence
\begin{align*}
0 \lrt & \ \ \ \ \SM(m, m_{n+1}) \  \lrt  \ \ \SM(m, m_n) \  \lrt \ \cdots \lrt \  \SM(m, m_1) \ \  \lrt \ \SM(m, m_0)\\
{}\lrt & \ \ \Ext^n_\La(m, m_{n+1}) \lrt  \Ext^n_\La(m, m_n) \ \lrt  \cdots \lrt  \Ext^n_\La(m, m_1) \ \lrt \Ext^n_\La(m, m_0)\\
{}\lrt & \ \Ext^{2n}_\La(m, m_{n+1})  \lrt  \Ext^{2n}_\La(m, m_n)\lrt  \cdots\lrt  \Ext^{2n}_\La(m, m_1)\lrt \Ext^{2n}_\La(m, m_0)\\
{} \lrt & \cdots
\end{align*}
of abelian groups.

By \cite[Definition-Proposition  2.15]{IJ}, $\SM$ is an $n\Z$-cluster tilting subcategory of $\mmod\La$ if and only if it is closed under $n$-syzygies, i.e. $\Omega_n\SM \subseteq \SM$, or equivalently, it is closed under $n$-cosyzygies, i.e. $\Omega_{-n}\SM \subseteq \SM$.

\begin{definition}\label{nZhomological}
Let $\La$ be an Artin algebra. We say that $(\La, \SM)$ is an $n\Z$-homological pair if $\SM$ is an $n\Z$-cluster tilting subcategory of $\mmod\La$.
\end{definition}

\s{\sc Cotorsion theory.} The idea of cotorsion pairs goes back to the work of Salce on abelian groups \cite{S}. Let us recall the notion in abelian categories. Let $\SA$ be an abelian category with enough projectives and injectives.  Let $\SX$ and $\SY$ be  full subcategories  of $\SA$. $(\SX, \SY)$ is called a cotorsion pair 
\cite[Definition 2.2.1]{GT} if it satisfies the following conditions.

\begin{itemize}
\item[$(i)$] $\SX={}^{\perp_1}\SY$, where ${}^{\perp_1}\SY=\{ x \in\SA \vert\  \Ext^1_\SA(x, y)=0, \ \mbox{for all} \  \ y \in \SY\}$,
\item[$(ii)$] $\SX^{\perp_1}=\SY$, where $\SX^{\perp_1}=\{ y \in\SA \vert \ \Ext^1_\SA(x, y)=0,  \ \mbox{for all}\  \ x \in \SX\}$.
\end{itemize}
$\SX$ is called a cotorsion class and $\SY$ is called a cotorsion free class.

The cotorsion pair $(\SX, \SY)$ is called complete if for every object $a \in \SA$ there exist short exact sequences
\[0 \lrt y \lrt x \lrt a \lrt 0 \ \ \ \ \ {\rm and } \ \ \ \ \ 0 \lrt a \lrt y \lrt x \lrt 0 \]
in $\SA$ such that $x \in \SX$ and $y \in \SY$.

\s A class $\SX$ of objects of $\SA$ is called a special precovering class if every object $a\in\SA$ admits a special right $\SX$-approximation, i.e. there exists a short exact sequence \[0\lrt y \lrt x \lrt a \lrt 0\] in $\SA$ such that $x\in\SX$ and $y \in \SX^{\perp_1}$.
Dually one can define the notion of special preenveloping classes. Hence a cotorsion pair $(\SX, \SY)$ is complete if $\SX$ is a special precovering class and $\SY$ is a special preenveloping class of $\SA$.  Salce's Lemma \cite{S} states that in a cotorsion pair $(\SX, \SY)$ in the category $\Mod R$, where $R$ is an associative ring with identity, $\SX$ is a special precovering class if and only if $\SY$ is a special preenveloping class.

\section{Higher Cotorsion Classes}
Let $\SM$ be an $n$-cluster tilting subcategory of an abelian category $\SA$ and $\SX \subseteq \SM$ be a full subcategory. Let ${\SX\mbox{-}{\rm exact}}_n$ denote the class of all sequences
\[ m_n \lrt m_{n-1} \lrt \cdots \lrt m_1\]
in $\SM$ with the property that for all $x \in \SX$,  the induced sequence
\[ 0 \lrt \Ext^n_{\SA}(x, m_n) \lrt \Ext^n_{\SA}(x, m_{n-1}) \lrt \cdots \lrt \Ext^n_\SA(x, m_1) \lrt 0\]
of abelian groups is exact.

Moreover, let ${}^{\perp}({{\SX\mbox{-}{\rm exact}}_n})$ be the full subcategory of $\SM$ consisting of all $m \in \SM$ such that for every sequence
 \[ m_n \lrt m_{n-1} \lrt \cdots \lrt m_1\]
in ${\SX\mbox{-}{\rm exact}}_n$ the induced sequence
\[ 0 \lrt \Ext^n_{\SA}(m, m_n) \lrt \Ext^n_{\SA}(m, m_{n-1}) \lrt \cdots \lrt \Ext^n_\SA(m, m_1) \lrt 0\]
of abelian groups is exact.

Obviously $\SX \subseteq {}^{\perp}({{\SX\mbox{-}{\rm exact}}_n})$.

\begin{definition}
With the above notations, we say that $\SX$ is an $n$-cotorsion class if 
\[\SX = {}^{\perp}({{\SX\mbox{-}{\rm exact}}_n}).\]
\end{definition}

Clearly if we let $n=1$, we recover the classical case, i.e. by abuse of notation, $(\SX, {\SX\mbox{-}{\rm exact}}_1)$ is a (classical) cotorsion pair.

The following proposition lists some of the basic properties of $n$-cotorsion classes. It, in particular, implies that $\SX$ is an additive subcategory of $\SM$ in the sense of \cite[Definition 2.7]{HJV}. 

\begin{proposition}\label{BasicProperty}
Let $\SX \subseteq \SM$ be an $n$-cotorsion class. Then the following holds.
\begin{itemize}
\item[$(i)$] $\SX$ is closed under direct summands and direct sums.
\item[$(ii)$] $\SX$ contains projective objects.
\item[$(iii)$] If $\SX$ is closed under $n$-syzygies, then for every sequence $r_n \rt \cdots \rt r_1$ in $\SX\mbox{-}{{\rm exact}}_n$ and for every $x \in \SX$, there exists an exact sequence \[0 \lrt \Ext^{ni}_{\SA}(x, r_n) \lrt \cdots \lrt \Ext^{ni}_{\SA}(x, r_1) \lrt 0,\] of abelian groups, for every integer $i\geq 1$.
\end{itemize}
\end{proposition}

\begin{proof}
Statements $(i)$ and $(ii)$ follows from definition. For the statement $(iii)$, let $x \in \SX$ be given. Since $\SX$ is closed under $n$-syzygies, $\Omega_{ni}x \in \SX$, for all $i \geq 1$. So we have the exact sequence of $\Ext^n$-groups for $\Omega_{ni}x \in \SX$. This implies that we have the exact sequence of $\Ext^{ni}$-groups for $x$.
\end{proof}

\begin{definition}\label{n-special precover}
Let $\SX$ be a full subcategory of $\SM$. We say that $\SX$ is an $n$-special precovering class if for every $m\in\SM$, there exists an $n$-exact sequence 
\[0 \lrt r_n\lrt \cdots\lrt r_1\lrt x \st{\va}\lrt m\lrt 0\] such that $x\in\SX$ and $r_n\lrt \cdots\lrt r_1\in {\SX\mbox{-}{\rm exact}_n}$. $\va:x \lrt m$ is called an $n$-special precover of $m$. 
\end{definition}

\begin{remark}\label{NspecialPrecovering}
Let $\SX \subseteq \SM$ be an $n$-special precovering class in an $n$-homological pair $(\La, \SM)$. Then $\SX$ is a precovering class.
To see this, note that since $\SX$ is an  $n$-special precovering class, then for every $m\in\SM$, there exists an $n$-exact sequence
\[0 \lrt r_n\lrt \cdots\lrt r_1\lrt x\st{\va}\lrt m\lrt 0\]
such that $x\in\SX$ and $r_n\lrt \cdots\lrt r_1\in {\SX\mbox{-}{\rm exact}_n}$. Note that $\va$ is an $\SX$-precover of $m$. To see this, let $x'\rt m$ be a morphism in $\SM$ with $x'\in\SX$. Consider the exact sequence
\[\cdots\rt \SM(x', x)\st{\va^*}\rt \SM(x', m)\rt \Ext^n_\La(x', r_n)\st{\psi}\rt \Ext^n_\La(x', r_{n-1})\]
of abelian groups. Since $r_n \lrt \cdots\lrt r_1\in{\SX\mbox{-}{\rm exact}_n}$, $\psi$ is injective. Hence $\va^*$ is surjective. Therefore $x' \lrt m$ factors through $\va$.  So $\va$ is an $\SX$-precover of $m$. Since $m$ was arbitrary, it implies that $\SX$ is a precovering class.
\end{remark}

\begin{theorem}\label{NspecialNcotor}
Let $(\La, \SM)$ be an $n$-homological pair. Let $\SX \subseteq \SM$ be an $n$-special precovering class which is closed under direct summands. Then $\SX$ is an $n$-cotorsion class.
\end{theorem}

\begin{proof}
To show that $\SX$ is an $n$-cotorsion class, by definition, we should show that $ {}^{\perp}({{\SX\mbox{-}{\rm exact}}_n}) \subseteq \SX$. To this end, assume that $m\in {}^{\perp}({{\SX\mbox{-}{\rm exact}}_n})$ is given. Since $\SX$ is an $n$-special precovering class, there exists an $n$-exact sequence
\[0 \lrt r_n \lrt \cdots \lrt r_1 \lrt x\st{\va} \lrt m \lrt 0\]
such that $x\in\SX$ and $r_n \lrt \cdots \lrt r_1 \in {{\SX\mbox{-}{\rm exact}}_n}$.  Consider the exact sequence
\[\cdots \lrt \SM(m, x)\st{\va^*} \lrt \SM(m, m) \lrt \Ext^n_\La(m, r_n)\st{\psi} \lrt \Ext^n_\La(m, r_{n-1})\]
of abelian groups. Now $m\in {}^{\perp}({{\SX\mbox{-}{\rm exact}}_n})$ implies that $\psi$ is injective. So $\va^*$ is surjective and hence $\va$ is a split epimorphism.  Therefore $m$ is a summand of $x$. Hence $m\in\SX$, as $\SX$ is closed under summands. So the proof is complete.
\end{proof}

\begin{remark}
Let $\prj\La$ denotes the class of projective $\La$-modules. It is obviously an $n$-special precovering class and is closed under summands. Hence, by the above theorem, it is an $n$-cotorsion class in the $n$-homological pair $(\La, \SM)$.
\end{remark}

\begin{theorem}
Let $(\La, \SM)$ be an $n\Z$-homological pair.  Let $\SX \subseteq \SM$ be a special precovering class which is closed under direct summands and satisfies the condition $\Ext^n_{\La}(\SX, \SX)=0$. Then $\SX$ is an $n$-cotorsion class.
\end{theorem}

\begin{proof}
By Theorem \ref{NspecialNcotor}, it is enough to show that $\SX$ is an $n$-special precovering class of $\SM$. Assume that $m \in \SM$ is given. Since $\SX$ is a special precovering class, there exists a short exact sequence $\eta: 0\lrt y \lrt x \st{\va} \lrt m \lrt 0$, with $x \in\SX$ and $y \in\SX^{{\perp}_1}$. By taking $n$-kernel of $\va$ in $\SM$, we get the following $n$-exact sequence
\[0\lrt r_n \lrt \cdots \lrt r_1 \lrt x \lrt m \lrt 0.\]
To complete the proof, we need to show that $r_n \lrt \cdots \lrt r_1$ is in $\SX\mbox{-}{\rm exact}_n$. We do this. Since $\SM$ is $n\Z$-cluster tilting, for every $x'\in\SX$, we have exact sequence
\[\cdots \rt \SM(x', x)\st{\va^*} \rt \SM(x', m) \rt \Ext^n_\La(x', r_n) \rt \cdots \rt \Ext^n_\La(x', r_1) \rt \Ext^n_\La(x', x) \st{\va_*} \rt \Ext^n_\La(x', m)\]
of abelian groups.

The long exact sequence of Ext groups inducing from the short exact sequence $\eta$, in view of the facts that  $\Ext^1_\La(x', y)=0$, implies that
$\SM(x', x)\st{\va^*}\lrt \SM(x', m)$ is an epimorphism. Moreover, by assumption, $\Ext^n_\La(x', x)=0$. Hence we get the desired exact sequence
\[0\lrt \Ext^n_\La(x', r_n) \lrt \cdots \lrt \Ext^n_\La(x', r_1) \lrt 0.\]
Therefore $r_n \lrt \cdots \lrt r_1$ is in $\SX\mbox{-}{\rm exact}_n$. The proof is hence complete.
\end{proof}

As an immediate corollary of the above theorem, we have the following result.

\begin{corollary}
Let $(\La, \SM)$ be an $n\Z$-homological pair. Let $(\SX, \SY)$ be a complete cotorsion pair in $\mmod\La$ such that $\SX \subseteq \SM$ and satisfies the condition $\Ext^n_{\La}(\SX, \SX)=0$. Then $\SX$ is an $n$-cotorsion class and $\SM \subseteq \SY.$
\end{corollary}

\begin{proof}
It follows from the above theorem that $\SX$ is an $n$-cotorsion class. To see that $\SM \subseteq \SY$, let $m \in \SM$ is given. Consider the short exact sequence 
$0 \lrt m \lrt y \lrt x \lrt 0$, with $x \in \SX$ and $y \in \SY$. Since $\SX \subseteq \SM$, $\Ext^1(x, m)=0$. Hence the sequence is split and so $m$ is a summand of $y$. Since $\SY$ is closed under summands, $m \in \SY$. So $\SM \subseteq \SY.$
\end{proof}

\section{Wakamatsu's Lemma}
An important result in classical homological algebra that provides a sufficient condition for the existence of special precovers is Wakamatsu's Lemma \cite[Lemma 2.1.1]{X}. Wakamatsu's Lemma states that for a given subclass $\SX$ of $\SA$ which is closed under summands and extensions, every surjective $\SX$-cover of an object $a \in \SA$ is a special $\SX$-precover, i.e. its kernel belongs to $\SX^{\perp}$. See also \cite[Lemma 2.1.13]{GT}. Recall that $\SX$ is closed under extensions if for every short exact sequence $0 \rt X' \rt A \rt X \rt 0$ in $\CA$ with $X, X' \in \SX$, we deduce that $A \in \SX$. For a version of Wakamatsu's Lemma for $(n+2)$-angulated categories, see \cite[Lemma 3.1]{Jor}. In the following, we present a version of this lemma in $n$-abelian categories.

Let $\SM$ be an $n$-abelian category and $\SX$ be a full subcategory of $\SM$. By \cite[Appendix A]{I1} an $n$-exact sequence
\[0 \lrt m_{n+1} \st{ u_{n+1} }{ \lrt }m_n \lrt  \cdots \lrt m_1 \st{ u_1 }{\lrt } m_{0} \lrt 0\]
in $\SM$ is called  almost-minimal if for any $2\leq i \leq n$, $u_i$ is in the Jacobson radical of $\SM$. Recall that the Jacobson radical of $\SM$ is a two sided ideal in $\SM$, denoted by $\rad_{\SM}$, and for $m$ and $m'$ in $\SM$, is defined by
\[\rad_{\SM}(m,m')=\{f: m \lrt m' \mid \ 1_m-gf \ {\rm is \ invertible \ for \ every } \ g: m' \lrt m\}.\]
  
\begin{definition}
Let $\SM$ be an $n$-abelian category and $\SX$ be a full subcategory of $\SM$. We say that $\SX$ is left closed under $n$-extensions if for each  almost-minimal $n$-exact sequence
\[0 \lrt x \lrt x_n \lrt x_{n-1} \lrt \cdots \lrt x_1 \lrt x' \lrt 0\]
with $x$ and $x'$ in $\SX$, we deduce that $x_n \in \SX$.
\end{definition}

Let $(\La, \SM)$ be an $n$-homological pair. Let $\SX \subseteq \SM$ be an additive subcategory. $\SX$ is closed under $n$-extensions \cite[Definition 2.10]{Fe} if every $n$-exact sequence
\[0 \lrt x' \lrt m_n \lrt m_{n-1} \lrt \cdots \lrt m_0 \lrt x \lrt 0\]
in $\SM$ with $x, x' \in \SX$ is Yoneda equivalent to an $n$-exact sequence 
\[0 \lrt x' \lrt x_n \lrt x_{n-1} \lrt \cdots \lrt x_0 \lrt x \lrt 0\]
with all terms in $\SX$. 
Now let $\SX \subseteq \SM$ be closed under $n$-extensions and let $x \lrt m$ be the $\SX$-cover of an object $m$ in $\SM$. It is proved in \cite[Lemma 5.1]{Fe} that for every $x' \in \SX$ the morphism \[ \Ext^n_{\La}(x', x) \lrt \Ext^n_{\La}(x', m)\]
is a monomorphism of abelian groups. Here we use this fact in the proof of next theorem. For the convenience of the reader we provide a proof.

\begin{theorem}(Wakamatsu's Lemma)\label{Wakamatsu}
Let $(\La, \SM)$ be an $n\Z$-homological pair. Let $\SX\subseteq \SM$ be a full subcategory which is left closed under $n$-extensions. Let $m \in \SM$ and $f: x \lrt m$ be a surjective $\SX$-cover of $m$. Then for every $n$-kernel $k_n \lrt \cdots \lrt k_1$ of $f$ and every $x' \in \SX$, there is an exact sequence
\[0\rt \Ext^{n}_\La(x', {k_n})\rt \cdots\rt\Ext^{n}_\La(x', {k_1})\rt 0,\]
of abelian groups.
\end{theorem}

\begin{proof}
Take $n$-kernel of $f$ to get the $n$-exact sequence
\[0 \lrt k_n \lrt \cdots \lrt k_1 \lrt x \st{f} \lrt m \lrt 0.\]
Since $\SM$ is an $n\Z$-cluster tilting subcategory of $\mmod\La$, for every $x' \in \SX$, we have long exact sequence
\begin{align*}
0 \lrt \ \SM(x', k_n) \ \lrt \cdots \lrt \ \SM(x', k_1) \ \ \lrt \  \SM(x', x) \ \  \st{f^*}\lrt & \ \SM(x', m) \ \lrt \\
 \Ext^n_\La(x', k_n) \lrt  \cdots \lrt \Ext^n_\La(x', k_1) \lrt \Ext^n_\La(x', x) \st{\hat{f}}\lrt  & \Ext^n_\La(x', m) \lrt \cdots
\end{align*}
of abelian groups. For the proof, it is enough to show that $f^*$ is surjective and $\hat{f}$ is injective. Since $f$ is an $\SX$-cover and $x' \in \SX$, it follows that  $f^*$ is surjective. We show that $\hat{f}$ is injective as well \cite[Lemma 5.1]{Fe}. Let
\[\eta: 0 \lrt x \lrt x_n \lrt \cdots \lrt x_1 \lrt x' \lrt 0\]
be an object of $\Ext^n_\La(x', x)$ that maps to zero in $\Ext^n_\La(x', m)$, that is, the $n$-pushout of $\eta$ along $f$, say $\eta'$, is contractible. We show that $\eta$ itself should be contractible. We note that by A.1.Proposition of \cite{I1} we can assume that $\eta$ is almost-minimal.  Assume that $\eta'$ illustrated in the following $n$-pushout diagram is contractible.
\begin{equation*}
\begin{tikzcd}
\eta\dar{\hat{f}}: &x \dar{f} \rar{d^n_x} & x_n \rar{d^{n-1}_x} \dar{f^n}\dlar[dotted] & x_{n-1} \rar{d^{n-2}_x}\dar{f^{n-1}} & \cdots \rar & x_1 \rar[two heads]{d^0_x} \dar{f^1} &  x' \dar[equals]\dlar[dotted]\\
\eta': &m \rar{d^n_y} & y_n \rar{d^{n-1}_y} & y_{n-1} \rar{d^{n-2}_y} & \cdots \rar & y_1 \rar[two heads]{d^0_y}  & x'
\end{tikzcd}
\end{equation*}
 So there exists morphism  $s^{n+1}: x'\rt y_1$ such that $d^{0}_y s^{n+1}=1_{x'}$.  This implies, by \cite[ Lemma 3.6]{Fe}, that morphism $\hat{f}:\eta\rt \eta'$ is null-homotopic. In particular, there exists a morphism $s^1: x_n\rt m$ such that $s^1d^n_x=f$. Now since $\SX$ is left closed under $n$-extensions, $x_n\in\SX$ and so $s^1$ factors through $f$. Let $g: x_n\rt x$ be such that $fg= s^1$.  By composition with $d^n_x$  we get $fg d^n_x=s^1d^n_x=f$. Since $f$ is an $\SX$-cover, $gd^n_x: x\rt x$ is an isomorphism. So $d^n_x$ is a split monomorphism. Hence $\eta$ is a contractible $n$-exact sequence, by part $(i)$ of Remark \ref{Contractible}.
\end{proof}

\begin{corollary}
Let $(\La, \SM)$ be an $n\Z$-homological pair. Let $\SX\subseteq \SM$ be a full subcategory which is closed under direct summands, left closed under $n$-extensions and permits surjective $\SX$-covers. Then $\SX$ is an $n$-cotorsion class.
\end{corollary}

\begin{proof}
In view of Wakamatsu's Lemma, we deduce that $\SX$ is an $n$-special precovering class. Hence by Theorem \ref{NspecialNcotor} it is an $n$-cotorsion class.
\end{proof}

\section{Wide subcategories}
In this section we provide some examples of $n$-cotorsion classes in an $n$-abelian category, using the notion of wide subcategories in higher homological algebra introduced and studied in \cite{HJV}. Let us begin with the definition of wide subcategories.

\begin{definition}\label{DefWide}(See \cite[Definition 2.8]{HJV})
Let $\SM$ be an $n$-abelian category. Let $\SW$ be an additive subcategory of $\SM$, i.e. is a full subcategory closed under direct sums, direct summands, and isomorphisms. $\SW$ is called wide if the following conditions hold.
\begin{itemize}
\item[$(i)$] Every morphism in $\SW$ admits an $n$-kernel and an $n$-cokernel in $\SM$ with all terms in $\SW$.
\item[$(ii)$] Every $n$-exact sequence \[0 \lrt w \lrt m_n \lrt \cdots \lrt m_1 \lrt w' \lrt 0,\]  with $w, w' \in \SW$ in $\SM$, is Yoneda equivalent to an $n$-exact sequence \[0 \lrt w \lrt w_n \lrt \cdots \lrt w_1 \lrt w' \lrt 0,\] with all terms in $\SW$.
\end{itemize}
\end{definition}

\begin{theorem}\label{Wide}
Let $(\La, \SM)$ be an $n\Z$-homological pair. Let $\SW$ be a wide subcategory of $\SM$ containing projective modules. If the inclusion functor $\SW\st{i} \lrt \SM$ has a right adjoint, then $\SW$ is an $n$-cotorsion class in $\SM$.
\end{theorem}

\begin{proof}
Since the inclusion functor $\SW\st{i} \lrt \SM$ admits a right adjoint, by dual of the Proposition 4.5 of \cite{HJV}, we conclude that $\SW$ is a covering subcategory of $\SM$.  In order to show that $\SW$ is an $n$-cotorsion class, we follow Theorem \ref{NspecialNcotor} and show that $\SW$ is an $n$-special precovering class.  Let $m\in\SM$ be given. Since $\SW$ is a covering class, there exists a surjective  $\SW$-cover $\va: w \lrt m$ of $m$. Consider the $n$-exact sequence
\[0 \lrt  r_n \lrt \cdots \lrt r_1 \lrt w \st{\va} \lrt m \lrt 0,\]
where $r_n \lrt \cdots\lrt r_1$ is an $n$-kernel of $\va$ in $\SM$. We should show that $r_n \lrt \cdots \lrt r_1$ lies in ${\SW\mbox{-}{\rm exact}}_n$, i.e. show that for every $w' \in \SW$, there exists an exact sequence
\begin{equation}\label{Exact}
0 \lrt \Ext^n_\La(w', r_n) \lrt \Ext^n_{\La}(w', r_{n-1}) \lrt \cdots \lrt \Ext^n_\La(w', r_1) \lrt 0
\end{equation}
of abelian Ext groups. Since $\SM$ is an $n\Z$-cluster tilting subcategory, there is the following long exact sequence
\begin{align*}
0 \lrt \SM(w', r_n)  \ \lrt  \cdots \lrt \ \SM(w', r_1) \ \ \lrt \  \SM(w', w) \ \  \st{\va^*}\lrt & \ \ \SM(w', m) \ \lrt \\
 \Ext^n_\La(w', r_n) \lrt  \cdots \lrt \Ext^n_\La(w', r_1) \lrt \Ext^n_\La(w', w) \st{\va_*}\lrt  & \ \Ext^n_\La(w', m) \lrt \cdots
\end{align*}
of abelian groups. We get the desired exact sequence \ref{Exact} if we show that ${\va^*}$ is an epimorphism and ${\va_*}$ is a monomorphism. Since $\va$ is a $\SW$-cover, $\va^*$ is surjective. ${\va_*}$ is a monomorphism using \cite[Lemma 5.1]{Fe}, that is, following the same argument as in the last part of the proof of Theorem \ref{Wakamatsu}. This completes the proof.
\end{proof}

Let $F: \SD \lrt \SD'$ be a functor between categories $\SD$ and $\SD'$. Then the essential image of $F$, denoted by $F(\SD)$, is the full subcategory
\[\{ d' \in \SD' \vert \ d' \cong F(d),\ \  \mbox{for some}\ d \in \SD\}\]
of $\SD'$, see e.g. \cite[Definition 3.7]{HJV}.

\begin{definition}\label{MorphismOfNhomological}(\cite[Definition 5.1]{HJV})
Let $(\La, \SM)$ and $(\La', \SM')$ be two $n$-homological pairs. We say that $(\La, \SM)\st{\phi}\lrt (\La', \SM') $ is an epimorphism of $n$-homological pairs if $\La\st{\phi} \lrt \La'$ is an algebra homomorphism such that $\phi_*(\SM') \subseteq \SM$, where $\phi_*:\mmod\La'\lrt \mmod\La$ is the functor given by restriction of scalers, and algebra homomorphism $\phi$ is an epimorphism in the category of rings.
\end{definition}

Let $\phi_*:\mmod\La' \lrt \mmod\La$ be an exact functor. For $n', n'' \in \mmod\La'$ and $j\geq 1$, there exists an induced homomorphism of Yoneda Ext groups,
\[\begin{tikzcd}
\Ext^j_{\La'}(n', n'') \ar{rr}{\phi_*(-)} && \Ext^j_{\La}(\phi_*(n'), \phi_*(n''))
\end{tikzcd}\]
such that $\phi_*(-)$ send the Yoneda equivalence class
\[0 \lrt n'' \lrt u_j \lrt \cdots \lrt u_1 \lrt n' \lrt 0\]
 of $\Ext^j_{\La'}(n', n'')$ to Yoneda equivalence class
\[0 \lrt \phi_*(n'') \lrt \phi_*(u_n) \lrt \cdots \lrt \phi_*(u_1) \lrt \phi_*(n') \lrt 0,\]
of $\Ext^j_\La((\phi_*(n'), \phi_*(n''))$. See \cite[Remark 3.4]{HJV}.\\

\begin{proposition}
Let $(\La, \SM)\st{\phi} \rt (\La' , \SM') $ be an epimorphism of $n$-homological pairs and $\phi_*:\mmod\La' \rt \mmod\La$ be the functor given by restriction of scalers. Assume that
\begin{itemize}
\item[$(i)$] For $m', m''\in\SM'$, the induced homomorphism $\Ext^n_{\La'}(m', m'')\st{\phi_*(-)} \lrt  \Ext^n_\La(\phi_*(m'), \phi_*(m''))$ of Yoneda Ext groups is bijective.
\item[$(ii)$] $\SX'$ is an $n$-cotorsion class in $\SM'$.
\end{itemize}
 Then $\phi_*(\SM')$ is an $n$-abelian category and $\phi_*(\SX')\subseteq \phi_*(\SM')$ is an $n$-cotorsion class.
\end{proposition}

\begin{proof}
Since $\phi$ is an epimorphism of $n$-homological pair, by \cite[Proposition 5.6]{HJV}, the induced functor $\SM' \st{\phi_*}\lrt \SM$ is full and faithful. So $\phi_*$ satisfies the condition $(a)$ of Theorem 3.8 of \cite{HJV}. Condition $(i)$ is just part $(b)$ of  \cite[Theorem 3.8]{HJV}. Hence the essential image $\phi_*(\SM')$ is a wide subcategory of $\SM$.
So by  Proposition 4.2 of \cite{HJV} $\phi_*(\SM')$ is an $n$-abelian category. To show that $\phi_*(\SX')\subseteq\phi_*(\SM')$ is an $n$-cotorsion class, by definition, we need to show the validity of the inclusion ${}^{\perp}({{\phi_*(\SX')\mbox{-}{\rm exact}}_n})\subseteq \phi_*(\SX') $.
Let $m\in {}^{\perp}({{\phi_*(\SX')\mbox{-}{\rm exact}}_n})$, where orthogonal is taken in $\phi_*(\SM')$.  We have $m=\phi_*(m')$, for some $m' \in \SM'$. So for every sequence 
\[\phi_*(m'_n) \lrt \cdots \lrt \phi_*(m'_1)\] in $\phi_*(\SX')\mbox{-}{\rm exact}_n,$
there is exact sequence
\[0 \lrt \Ext^n_\La(\phi_*(m'),\phi_*(m'_n)) \lrt \cdots \lrt \Ext^n_\La(\phi_*(m'), \phi_*(m'_1)) \lrt 0.\]
Hence, condition $(i)$ implies that the sequence
\[0 \lrt \Ext^n_{\La'}(m', m'_n) \lrt \cdots \lrt \Ext^n_{\La'}(m', m'_1) \lrt 0\]
is also exact. Since $\SX'$ is an $n$-cotorsion class, we get $m' \in\SX'$. Hence $m=\phi_*(m')\in\phi_*(\SX')$. This completes the proof.
\end{proof}

 \section*{Acknowledgments}
The authors would like to express their special thanks to Professor Peter J{\o}rgensen for constant encouragement during this project, in particular, for suggesting this notion of $n$-cotorsion classes. This research is in part supported by a grant from IPM.


\begin{thebibliography}{29}
\bibitem[BR]{BR} {\sc A. Beligiannis and I. Reiten,} {\sl Homological and homotopical aspects of torsion theories,} Mem. Am. Math. Soc. {\bf 188}(883) (2007), viii+207.

\bibitem[EN]{EN} {\sc R. Ebrahimi and A. Nasr-Isfahani,}  {\sl Higher Auslander’s formula,} arXiv:2006.06472v1.

\bibitem[EJ]{EJ}{\sc E.E. Enochs and O.M.G. Jenda,} Relative Homological Algebra, de Gruyter Exp. Math. 30, Walter de Gruyter Co., 2000.

\bibitem[Fe]{Fe}{\sc F. Fedele,} {\sl $d$-Auslander-Reiten sequences in subcategories,} Proc. Edinburgh Math. Soc. {\bf 63} (2020), 342-373.

\bibitem[GT]{GT} {\sc R. G\"{o}bel,  and J. Trlifaj,} Approximations and Endomorphism Algebras of Modules: de Gruyter Expositions in Mathematics, 41, Walter de Gruyter, Berlin (2006).

\bibitem[HJV]{HJV}{\sc M. Hereschend, P. J{\o}rgensen and L. Vaso,} {\sl Wide subcategories of $d$-cluster tilting subcategories,}  Trans. Amer. Math. Soc., doi: 10.1090/tran/8051 (2020).
    
\bibitem[HS]{HS} {\sc P. J. Hilton and U. Stammbach,} {\sl A course in homological algebra,} Graduate Texts in Mathematics, Volume 4 (Springer-Verlag, New York, 1997).

\bibitem[I1]{I2} {\sc O. Iyama,} {\sl Auslander correspondence,} Adv. Math. {\bf 210} (2007) 51-82.

\bibitem[I2]{I1} {\sc O. Iyama,} {\sl Higher-dimensional Auslander-Reiten theory on maximal orthogonal subcategories,} Adv. Math. {\bf 210} (2007), 22-50.

\bibitem[I3]{I3} {\sc O. Iyama,} {\sl Cluster tilting for higher Auslander algebras,} Adv. Math. {\bf  226} (2011) 1-61.

\bibitem[IJ]{IJ} {\sc O. Iyama and G. Jasso,} {\sl Higher Auslander correspondence for dualizing R-varieties,} Algebr. Represent. Theory. {\bf 20} (2017), no 2, 335-354.

\bibitem[JJ]{JJ} {\sc K. M. Jacobsen and P. J{\o}rgensen,} {\sl $d$-abelian quotients of $(d + 2)$-angulated categories,} J. Algebra {\bf 521} (2019), 114-136.

\bibitem[Ja]{Ja} {\sc G. Jasso,} {\sl $n$-Abelian and $n$-exact categories,} Math. Z, {\bf 283} (2016), 703-759.

\bibitem[Jor]{Jor} {\sc P. J{\o}rgensen,} {\sl Torsion classes and t-structures in higher homological algebra,} Int. Math. Res. Not. IMRN {\bf 13} (2016), 3880-3905.

\bibitem[S]{S} {\sc L. Salce,} {\sl Cotorsion theories for abelian groups,} in: Symposia Mathematica, Vol. XXIII, Conf. Abelian Groups and their Relationship to the Theory of Modules, INDAM, Rome, 1977, Academic Press, London-New York, 1979, pp.11-32.
    
\bibitem[X]{X} {\sc J. Xu,} {\sl Flat Cover of Modules,} Lecture Notes in Math. 1634, Springer, New York, 1996.
\end{thebibliography}
\end{document}